\documentclass[11pt]{amsart}
\usepackage{mathrsfs}
\usepackage{latexsym,amsthm,amssymb,amsmath,amsfonts,graphicx,epstopdf,subfigure,float}
\usepackage{bbm}
\usepackage{pifont}
\usepackage{cite}
\usepackage{multirow}
\usepackage{appendix}
\usepackage{tikz,xcolor}

\usetikzlibrary{decorations.markings}
\usetikzlibrary{intersections}
\usetikzlibrary{calc}

\newtheorem{theorem}{Theorem}[section]
\newtheorem{thm}[theorem]{Theorem}

\newtheorem{lemma}{Lemma}[section]
\newtheorem{remark}{Remark}[section]
\newtheorem{defi}{Definition}[section]

\newtheorem{example}{Example}[section]
\numberwithin{equation}{section}

\def\N{\mathbb N}

 \usepackage{color}

\begin{document}

\title{Lattice subsequences of fixed points of Toeplitz substitutions}


\author{Shishuang Liu}
\address{Department of Mathematics and Statistics, Central China Normal University, Wuhan, 430079, China}
\email{shishuangliu@mails.ccnu.edu.cn}

\author{Hui Rao$^*$} \address{Department of Mathematics and Statistics, Central China Normal University, Wuhan, 430079, China}
\email{hrao@mail.ccnu.edu.cn}

\date{\today}


\thanks{* Corresponding author.}

\begin{abstract}
We define the modulo-$m$ Toeplitz fixed point generated by Toeplitz substitution and study the lattice subsequence of such fixed point. Moreover, we provide a method to check whether one modulo-$m$ Toeplitz fixed point is a lattice subsequence of another.
\end{abstract}
\maketitle

\section{\textbf{Introduction}}
Let $\mathcal A$ be a finite set, and  we call it an alphabet. Let $\mathcal A^*=\bigcup_{k\geq 0}{\mathcal A}^k$,
where $A^0=\{\varepsilon\}$ and $\varepsilon$ stands for the empty word.
Let ${\mathcal A}^\infty$ be the set of infinite sequence (or word) over ${\mathcal A}$.
Let $\gamma\tau$ be the concatenation of words $\gamma\in \mathcal A^*$ and $\tau\in \mathcal A^*\cup \mathcal A^\infty$.

Let $m\geq 2$ and let $W=w_1\dots w_{m-1}$ be a finite word over ${\mathcal A}$.
  We define the morphism $\sigma$ over ${\mathcal A}^*\cup {\mathcal A}^{\infty}$ by
 $$
 \sigma: a\mapsto Wa, \ \ a\in {\mathcal A},
 $$
 and call it a \emph{Toeplitz substitution} (determined by $W$).  Let $w_1$ be the initial letter of $W$, then
$$X=\lim_{k\to \infty} \sigma^{k}(w_1)$$
  is called the \emph{fixed point} of $\sigma$, and we call it a \emph{modulo-$m$ Toeplitz fixed point}.
 See \cite{CK97, G16 , QRWX10}.

 Toeplitz fixed point is a very special type of Toeplitz words. Toeplitz words were first introduced by Jacobs and Keane in \cite{JK69}. Jacobs and Keane  \cite{JK69} defined Toeplitz words and studied  the ergodic properties related to Toeplitz words. Prodinger and Urbanek \cite{PU79} and Sell \cite{S20}  studied  the  combinatorial properties of Toeplitz words.
Recently, Toeplitz words have gained increasing popularity as a model for quasicrystals, see for example \cite{BJL16}.   Liu and Qu \cite{LQ11,LQ12} studied the spectral properties of Schr\"odinger operator with Toeplitz potential.

\begin{defi}\emph{ Let $X=\{X(j)\}_{j\geq 1}$ and $Y=\{Y(j)\}_{j\geq 1}$ be two sequences. Let $q\geq 1$ be an integer. We say $Y$ is a \emph{$q$-subsequence} of $X$ if   $Y(j)=X(qj)$, $\forall j\geq 1$. In this case, we also say $Y$
is a \emph{lattice subsequence} of $X$.}
\end{defi}

In Huang, Liu and Rao \cite{HLR24}, to determine when two fractal sets are bi-Lipschitz equivalent, they arise the following question:

\medskip

\noindent \textbf{Question 1.} \emph{Let $X$ and $Y$ be two modulo-$m$ Toeplitz fixed points. How to judge whether
$Y$ is a $q$-subsequence of $X$?}

\medskip

The main goal of the present paper is to give an answer to the above question. From now on, we shall  use $X(q\N)$ to denote the $q$-subsequence of $X$. The following lemma is obvious.

\begin{lemma} Let $X$ and $Y$ be two modulo-$m$ Toeplitz fixed points.  Then $Y$ is a $q$-subsequence of $X$ if and only if $X(q\N)$ is a modulo-$m$ Toeplitz fixed point
and $Y(j)=X(qj)$ for $1\leq j\leq m-1$.
\end{lemma}

Therefore Question 1 is converted to
the following question:
\medskip

\noindent \textbf{Question 2.} \emph{Let $X$ be a modulo-$m$ Toeplitz fixed point.
For which $q$, $X(q\N)$ is also a modulo-$m$ Toeplitz fixed point?}

\medskip

It is easy to show that

\begin{lemma}\label{lem:coprime}
 Let $X=\{X(j)\}_{j\geq1}$ be a modulo-$m$ Toeplitz fixed point.  Let $h\geq 1$ be an integer such that  $\gcd\{m, h\}=1$.
 Let $q\geq 1$ be an integer. Then   $X({hq\N})$ is a  modulo-$m$ Toeplitz fixed point if and only if  $X(q\N)$ is a  modulo-$m$ Toeplitz fixed point.
\end{lemma}

Notice that  $X({mq\N})=X(q\N)$
 (see Lemma \ref{lem:basic}). Therefore, to characterize all lattice subsequences of a modulo-$m$ Teoplitz fixed point,
we just need to check the $q$-subsequence such that $q|m^s$ for some $s\geq 1$ and $m\nmid q$.

 Let $W=w_1\dots w_n$ be a word and let $2\leq q<n$. We say $W$ is \emph{almost $q$-periodic}, if  $q\nmid j$ implies $w_j=w_{j+q}$.
Similarly, if $W$ is an infinite word, we define $W$ is almost $q$-periodic in the same manner.

Let $a^{k}$ be the constant word with length $k$ and with letter $a$.  For each word $\omega$ we use $\omega(n)$ or $\omega_n$ to denote the letter at index $n$.

\begin{thm}\label{thm:main}
Let $X=\{X(j)\}_{j\geq 1}$   be a modulo-$m$ Toeplitz fixed point which is not a constant word. Let
$q\geq 2$ be an integer such that  $m\nmid q$ and $q|m^s$ for some $s\geq 1$.
Then the following hold:

(i) If $q\nmid m^2$, then $X(q\N)$ is not a modulo-$m$ Toeplitz fixed point.

 (ii) If $q| m$, then  $X(q\N)$   is a modulo-$m$ Toeplitz fixed point if and only
if  $X(1)\dots X(m)$ is almost $q$-periodic.

 (iii) If $q|m^2$, then  $X(q\N)$   is a modulo-$m$ Toeplitz fixed point if and only
if  $X(1)\dots X(m^2)$ is almost $q$-periodic.
\end{thm}

 \begin{remark}
\emph{
Let $X$ and $Y$ be two modulo-$m$ Toeplitz fixed points.  It is not hard to show that  if $q|m^s$ and $Y=X(q\N)$, then $X=Y(\frac{m^s}{q} \N)$.
}
\end{remark}

\section{\textbf{Lattice subsequences of Toeplitz words}}

  We shall use ${\mathcal J}_m$ to denote the collection of modulo-$m$ Teoplitz fixed points (over the alphabet ${\mathcal A}$).

\begin{lemma}\label{lem:basic}
A word  $X=\{X(j)\}_{j\geq 1}$  belongs to ${\mathcal J}_m$ if and only if the following two conditions hold:

(i) For any $j\geq 1$, $X(mj)=X(j)$;

(ii) If $m\nmid j$, then $X(j)=X(j+m)$.
\end{lemma}

\begin{proof} Suppose $X\in {\mathcal J}_m$. Denote $a=X(1)$. Then $\sigma^n(a)=\sigma(\sigma^{n-1}(a))$,
so $\sigma^n(a)$ is almost $m$-periodic. Therefore, $X$ is almost $m$-periodic, which confirms (ii).

Since $\sigma(X(1)\dots X(j))=\sigma(X(1))\sigma(X(2))\dots \sigma(X(j))$, we conclude that
$X(mj)$ is the last element of $\sigma(X(j))=WX(j)$, so $X(mj)=X(j)$ and (i) is confirmed.

On the other hand, suppose that (i) and (ii) hold. Denote $W=X(1)\dots X(m-1)$. Then for any $k\geq 1$, we have
$$X(1)\dots X(m^k)=\prod_{j=1}^{m^{k-1}} WX(mj)=\prod_{j=1}^{m^{k-1}} \sigma(j)=\sigma(X(1)\dots X(m^{k-1})),$$
which implies that $X\in {\mathcal J}_m$.
\end{proof}

\begin{proof}[\textbf{Proof of Lemma \ref{lem:coprime}}] Let $Y=X(hq\N)$ and $Z=X(q\N)$.
 Let $h\geq 2$.

 Suppose $Z\in {\mathcal J}_m$. First, it is obvious that
$$Y(mj)=X(m(hqj))=X(hqj)=Y(j).$$
Pick $j\geq 1$ such that $m\nmid j$. Since   $\gcd\{m,h\}=1$, we  have  $m\nmid hj$. Therefore,
$$
Y(j)=Z(hj)=Z(hj+hm)=Z(h(j+m))=Y(j+m).
$$
By Lemma \ref{lem:basic},  we obtain that $Y$ is a modulo-$m$ Teoplitz fixed point.

 Suppose $Y\in {\mathcal J}_m$. Let $p\geq 1$ be an integer such that $ph\equiv 1 \pmod m$. Then $\gcd\{p, m\}=1$.
 By what we have just proved, $Y(p\N)\in \mathcal J_m$.
 Pick $j\geq 1$. Write $ph=ym+1$ and $qj=m^tj'$ where $m\nmid j'$. Then
  $$Y(pj)=X(phqj)=X((ym+1)m^tj')=X(j')=X(m^tj')=Z(j).$$
  Hence $Z=Y(p\N)\in \mathcal J_m$.
\end{proof}

\medskip

%


Now we introduce words with question mark to study Teoplitz words. In other words, we study words over the alphabet
  $\mathcal A\cup\{?\}$.

For $U=u_1\dots u_{r-1}?$ and $V=v_1\dots v_{s-1}?$ where $u_j, v_i\in {\mathcal A}$ for $1\leq j\leq r-1$, $1\leq i\leq s-1$. We define \emph{the composition of $U$ and $V$} by
$$
U\circ V=\prod_{j=1}^s (u_1\dots u_{r-1} v_j),
$$
where $v_s=?$. Then all entries of $U\circ V$ belong to ${\mathcal A}$ except the last entry.
 By this reason, we can define the  \emph{$n$-th iteration of $U$} by
$$
(U)^{(1)}=U \text{ and } (U)^{(n)}=(U)^{(n-1)}\circ U,\ \forall n\geq 2.
$$
Clearly the operation `$\circ$' is not commutative but it is associative, namely,
$$(U\circ V)\circ W=U\circ(V\circ W),$$
where $U=u_1\dots u_{r-1}?$, $V=v_1\dots v_{s-1}?$ and $W=w_1\dots w_{t-1}?$.
Moreover, for $\boldsymbol{\omega}=(\omega_j)_{j=1}^\infty\in {\mathcal A}^\infty$, we define
$$
U\circ \boldsymbol{\omega}= \prod_{j=1}^\infty (u_1\dots u_{r-1} \omega_j).
$$

The following lemma is obvious.

\begin{lemma}If $X=\{X(j)\}_{j\geq 1}$ is a modulo-$m$ Teoplitz fixed point. Then
$$
X=\lim_{n\to \infty} (X(1)\dots X(m-1)?)^{(n)}:=[X(1)\dots X(m-1)?]^{(\infty)}.
$$
\end{lemma}

\begin{thm}\label{thm:lattice}
Let $X=\{X(j)\}_{j\geq 1}\in {\mathcal J}_m$. Let
$q\geq 2$ be an integer such that $m\nmid q$ and  $q|m^s$  for some integer $s\geq 1$.
Then the following three statements are equivalent:

(i) $X(q\N) \in {\mathcal J}_m$.

(ii) $X$ is almost $q$-periodic, in other words,
\begin{equation*}\label{eq:comp}
X=(X(1)\dots X({q-1})?)\circ X(q\N).
\end{equation*}

(iii)  $X(1)\dots X(m^s-1)?=(X(1)\dots X(q-1)?)\circ (X(q)X(2q)\dots X(m^s-q)?).$
\end{thm}

\begin{proof}   Let $Y=X(q\N)$.

(i) $\Rightarrow$ (ii).  Since $Y$ is a modulo-$m$ Toeplitz word, we have
\begin{equation*}\label{eq:Ym}
  Y(j)=Y(j+tm^s), \quad\forall m\nmid j,   t\in \mathbb N,
\end{equation*}
therefore,
$$
Y(jm)=Y(jm+tm^{s+1}), \quad\forall m\nmid j, t\in \mathbb N,
$$
which implies
$$
X(jmq)=X(jmq+qtm^{s+1}), \quad\forall m\nmid j, t\in \mathbb N.
$$
Denote $p=m^s/q$. Pick  $1\leq \ell \leq q-1$
 and set $j=\ell p$,  then $jmq=\ell m^{s+1}$, so we obtain
$$
X(\ell)=X(\ell m^{s+1})=X(\ell m^{s+1}+qtm^{s+1})=X(\ell+tq),\quad t\in \mathbb N,
$$
which means that $X$ is almost $q$-periodic.

(ii)$\Rightarrow$(iii). Using (ii) to calculate the first $m^s$ terms of $X$, we obtain (iii).

(iii) $\Rightarrow$ (i).  Set $U=X(1)\dots X(q-1)?$ and $V=X(q)X(2q)\dots X((p-1)q)?$, we have
 $X(1)\dots X(m^s-1)?=U\circ V$. Therefore
 $$
 X=[X(1)\dots X(m^s-1)?]^{(\infty)}=(U\circ V)^{(\infty)}=U\circ (V\circ U)^{(\infty)},
 $$
 which implies that $X(q\N)=(V\circ U)^{(\infty)}\in {\mathcal J}_m$.

\end{proof}

\section{\textbf{The case $q\nmid m$}}
 In this section, we strengthen Theorem \ref{thm:lattice} by showing that if $X(q\N)$
 is a modulo-$m$ Teoplitz fixed point, then necessarily $q|m^2$. Moreover, in the case that
 $q\nmid m$, we  deduce that $X(1)\dots X(m)$ has a very special decomposition.

Let $X=\{X(j)\}_{j\geq 1}$ be a modulo-$m$ Toeplitz fixed point.
Let $q\geq 2$ be an integer such that $q|m^s$ for some integer $s \geq 2$, $m\nmid q$ and $q\nmid m$.
Then  $d=\gcd\{q, m\}>1$.  Denote $q_1=q/d$, $m_1=m/d$.

In what follows we always assume that the $q$-subsequence of $X$ is a modulo-$m$ Toeplitz fixed point.
\begin{lemma}\label{lem:almost-periodic}The following hold.

(i) $X$ is almost $q$-periodic.

(ii) $X$ is almost $d$-periodic.
\end{lemma}

\begin{proof}
(i) By Theorem \ref{thm:lattice}, $X$ is almost $q$-periodic, \text{i.e.},
$$X(j)=X(j+q), \quad \forall q\nmid j.$$

(ii) Pick $j_0\in \{1,\dots, d-1\}$. Write $d=xm-yq$  with $x, y\geq 1$.
Then $X(j_0+nd)=X(j_0+nxm-nyq)=X(j_0)$, since $X$ is almost $q$-periodic and almost $m$-periodic. The lemma is proved.
\end{proof}

\begin{lemma}\label{lem:j-j+1}
If $\emph{\text{lcm}}\{m,q\}\nmid jd$ and $\emph{\text{lcm}}\{m,q\}\nmid (j+1)d$ , then $
X(jd)=X((j+1)d).$
\end{lemma}

\begin{proof}  $q\nmid m$ implies that   $q_1\geq 2$.
Write $d=xm-yq$  with $x, y\geq 1$.

If $q\mid  jd$, then $m\nmid jd$ and   $q\nmid (j+1)d$,   we have
\begin{equation}\label{eq:case-1a}
X(jd)=X(jd+xm), \quad \text{(since $X$ is almost $m$-periodic)}
\end{equation}
and
\begin{equation}\label{eq:case-1b}
X((j+1)d)=X((j+1)d+yq), \quad \text{(since $X$ is almost $q$-periodic)}
\end{equation}
which imply $X(jd)=X((j+1)d)$.

If $m|jd$, then $m\nmid jd+d$ and $q\nmid jd$, we have
\begin{equation}\label{eq:case-1c}
X(jd)=X(jd+q)=X(jd+q+\frac{q-d}{d}yq)
\end{equation}
and
\begin{equation}\label{eq:case-1d}
X((j+1)d)=X((j+1)d+\frac{q-d}{d}xm),
\end{equation}
which imply $X(jd)=X((j+1)d)$.

Now suppose $q\nmid jd$ and $m\nmid jd$. If $q\nmid (j+1)d$, we use \eqref{eq:case-1a} and \eqref{eq:case-1b},
if $m\nmid (j+1)d$,   we use \eqref{eq:case-1c} and \eqref{eq:case-1d}.
\end{proof}

\begin{lemma}\label{lem:city} The following hold.

(i) If $jd$ is not a multiple of $\emph{\text{lcm}}\{m,q\}$, then $X(jd)=X(d).$

(ii) Let $j\in \{1,\dots, d-1\}$. If $q_1\nmid j$, then $X(j)=X(d)$.
\end{lemma}

\begin{proof} (i) We prove the assertion by induction on $j$. Suppose $\text{lcm}\{m,q\}\nmid jd$.

If $\text{lcm}\{m,q\}\nmid (j-1)d$, then  by Lemma \ref{lem:j-j+1}, $X(jd)=X((j-1)d)=X(d)$, where
the last equality is due to the induction hypothesis.

Now suppose $\text{lcm}\{m,q\}\mid (j-1)d$. Write $(j-1)d=xq=ym$.

If $q\nmid jd$, then $X(jd)=X(jd-xq)=X(d)$, since $X$ is almost $q$-periodic.

If $m\nmid jd$, then $X(jd)=X(jd-ym)=X(d)$, since $X$ is almost $m$-periodic. Item (i) is proved.

(ii)  Clearly $X(j)=X(jm)$. That $q_1\nmid j$ implies that $q\nmid jm$, hence, by  item (i),
$X(jm)=X(d)$.
\end{proof}

  By the following result, to check whether $X(q\N)$ is modulo-$m$ Teoplitz fixed point, we just need to examine the
word $X(1)\dots X(m)$.

\begin{thm}\label{thm:structure}
Let $X\in {\mathcal J}_m$  and assume that it is not a constant word. Let
$q\geq 2$ be an integer such that $q\nmid m$, $m\nmid q$ and   $q|m^s$  for some integer $s\geq 1$.
Denote $d=\gcd\{m,q\}$ and $q_1=q/d$.  Denote $a=X(1)$.
Then $X(q\N)\in {\mathcal J}_m$  if and only if $q|m^2$, $q_1$ is a proper factor of $d$ and $X(1)\dots X(m-1)?$ can be
decomposed into
\begin{equation}\label{eq:QTD}
X(1)\dots X(m-1)?=(a^{q_1-1}?)\circ (X(q_1)\dots X(tq_1)?)\circ (a^{m_1-1}?):=Q\circ T\circ D.
\end{equation}
  where $t=d/q_1-1$. In this case,
  $$X(q\N)=(D\circ T\circ Q)^{(\infty)}.$$
\end{thm}

\begin{proof}
First we prove the necessity. By Lemma \ref{lem:city}(i), we have
\begin{equation}\label{eq:d-seq}
  X(jd)=X(d), \quad \text{ for } 1\leq j\leq m/d.
\end{equation}
Especially, we have $X(1)=X(m)=X(d)$.

 By Lemma \ref{lem:almost-periodic}(ii),  $X(1)\dots X(m)$ is almost $d$-periodic, which together with  \eqref{eq:d-seq} imply that $X(1)\dots X(m)$ is $d$-periodic, i.e., $X(1)\dots X(m)=[X(1)\dots X(d)]^{m/d}.$
On the other hand, by Lemma \ref{lem:city}(ii) and \eqref{eq:d-seq}, $X(1)\dots X(d)$ can be decomposed into
$$X(1)\dots X(d)=[a^{q_1-1}?]\circ [X(q_1)\dots X(tq_1)a],$$
where $t=d/q_1-1$.
Therefore,
$$X(1)\dots X(m-1)?=(a^{q_1-1}?)\circ (X(q_1)\dots X(tq_1)?)\circ (a^{m_1-1}?):=Q\circ T\circ D.$$

Next, we  show that $q\mid m^2$ and $q_1$ is a proper factor of $d$.

If $q_1\geq d$, then by Lemma \ref{lem:city}(ii), $X(j)=a$ for all $j\in \{1,\dots, d-1\}$.
Since $X$ is almost $d$-periodic, we obtain that $X=a^\infty$. This contradiction proves that
$$q_1< d.$$

Suppose on the contrary that $q_1$ is not a proper factor of $d$. Pick $\ell\geq 1$ such that $ \ell q_1<d$.
Let  $j=\ell q_1+d$, then $q_1$ is not a factor of $d$ implies that $q_1\nmid j$, hence  $jm$ is not a multiple of $q_1m=mq/d$, so by Lemma \ref{lem:city}(i),
we have
$$X(j)=X(jm)=X(d)=a.$$
 On the other hand, since $X$ is almost $d$-periodic, we have $X(\ell q_1)=X(j)=a$.
 Therefore, for all $j\in \{1, \dots, d-1\}$, no matter $j$ is a multiple of $q_1$ or not, we always have $X(j)=a$.
 So $X$ is a constant word. A contradiction. Clearly $q=q_1d$ is a factor of $d^2$, and hence it is a factor of $m^2$.
The necessity  is proved.

Now we prove the sufficiency. Suppose \eqref{eq:QTD} holds. Using $D\circ Q=Q\circ D$, we obtain
$$
\begin{array}{rl}
X&=\lim_{k\to \infty} (Q\circ T\circ D)^{(k)}=Q\circ T\circ \lim_{k\to \infty} (D\circ Q\circ T)^{(k)}\\
&=Q\circ T\circ \lim_{k\to \infty} (Q\circ D\circ T)^{(k)}=Q\circ T\circ Q\circ \lim_{k\to \infty} (D\circ T\circ Q)^{(k)},
\end{array}
$$
which proves that $X(q\N)=\lim_{k\to \infty} (D\circ T\circ Q)^{(k)}$ is a modulo-$m$ Teoplitz fixed point.
The sufficiency is proved.
\end{proof}


%

\begin{example}
\emph{ Let  $m=12$. Let $X=[aabaaa aabaa?]^{(\infty)} $ be a modulo-$12$ Toeplitz fixed point. It is seen that
$$X=  [(aa?)\circ (b?)\circ(a?)]^{(\infty)}.$$}
\emph{Clearly $X(3\N)=[(b?)\circ(a?)\circ(aa?)]^{(\infty)} $ and $X(6\N)=[(a?)\circ(aa?)\circ(b?)]^{(\infty)}$.}

\emph{Let $q=18$, then  $d=\gcd(12,18)=6$, $q_1=q/d=3$ and $m_1=m/d=2$.
and
$$X(18\N)=  [(a?)\circ (b?)\circ(aa?)]^{(\infty)}.$$
Moreover, $X(q\N)\in {\mathcal J}_m$ if and only if $q=m^khp$ where $k\geq 0$, $\gcd\{h, m\}=1$ and $p\in \{1,3,6,18\}$.
One can show that  $X(q\N)$ coincides with one of $X$, $X(3\N)$, $X(6\N)$ and $X(18\N)$.
 }
\end{example}

\begin{example}
\emph{Let $m=6$. Let $X=[aaaba?]^{(\infty)}$. Then  }
 $$X(2\N)=[ababa?]^{(\infty)};$$
 $$X(5\N)=[abaaa?]^{(\infty)}.$$
\emph{Moreover, $X(q\N)\in {\mathcal J}_m$ if and only if $q=m^khp$ where $k\geq 0$, $\gcd\{h, m\}=1$ and $p\in \{1,2\}$;
 indeed, $X(q\N)$ coincides with one of  $X$, $X(2\N)$  and $X(5\N)$.}

\end{example}


\end{document}